\documentclass[a4paper,draft]{article}
\usepackage[utf8]{inputenc}
\usepackage{amsthm}
\usepackage{amsmath}
\usepackage{amssymb}
\usepackage{authblk}
\usepackage{caption}
\usepackage{commath}
\usepackage{mfabacus}
\usepackage{genyoungtabtikz}
\usepackage{tikz}
\usepackage{fancyhdr}

\theoremstyle{plain}
\newtheorem{thm}{Theorem}[section]

\theoremstyle{plain}
\newtheorem{prop}[thm]{Proposition}

\theoremstyle{plain}

\theoremstyle{plain}
\newtheorem{lemma}[thm]{Lemma}

\theoremstyle{plain}
\newtheorem{rem}[thm]{Remark}

\theoremstyle{plain}
\newtheorem{rems}[thm]{Remarks}

\theoremstyle{plain}

\theoremstyle{plain}

\theoremstyle{plain}
\newtheorem{definition}[thm]{Definition}

\theoremstyle{plain}
\newtheorem{exam}[thm]{Example}

\theoremstyle{plain}

\theoremstyle{plain}

\theoremstyle{plain}
\newtheorem{cor}[thm]{Corollary}

 \newcommand{\N}{{\mathbb N}}
\newcommand{\Z}{{\mathbb Z}} 
 
 \newcommand{\F}{{\mathbb F}}

 \newcommand{\He}{{\cal H}}

\begin{document}
\title{Generalised regularisation maps on partitions}

\author{Diego Millan Berdasco}
\affil{Queen Mary University of London\\ Mile End Road, \\ E1 4NS, UK\\ d.millanberdasco@qmul.ac.uk}
\date{}

\maketitle

\begin{abstract}
In a 1976 landmark paper, Gordon James defined the regularisation maps on integer partition, yielding certain decomposition numbers for modular representations of $\mathfrak{S}_n$. We describe a generalisation of James's regularisation map and give with proof an algorithm for such maps in the abacus.
\end{abstract}

\section{Introduction}
The Iwahori--Hecke algebras of the symmetric group $\He_{\F,q}(\mathfrak{S}_n)$ with quantum characteristic $e$ are examples of cellular algebras. Their cell modules are called Specht modules and are indexed by integer partitions $\lambda\vdash{n}$. All simple modules of $\He_{\F,q}(\mathfrak{S}_n)$ are isomorphic to certain quotients of Specht modules and indexed by $e$-regular integer partitions. Despite the nice properties derived from the cellular structure of the algebra, their representation theory is far from being well understood. The main open problem regarding Iwahori--Hecke algebras of type $A$ is the explicit description of the decomposition numbers $d_{\lambda\mu}$, which are the multiplicities of each simple module $D^{\mu}$ appearing as the quotients of the elements in the composition series of the Specht module $S^{\lambda}$. Knowing these composition series, or analogously the decomposition numbers, would give us a much better understanding of the structure of Specht modules, and thus of the whole algebra. For a outstanding account on Iwahori--Hecke algebras see [Ma].\newline
One of the earliest results in the literature regarding decomposition numbers was given by James in [J1]. He describes a combinatorial map on the set of partitions called $e$-regularisation, which for a given $\lambda$ returns the $e$-regular partition $\lambda^{reg_e}$ defining the smallest constituent in dominance order of $S^{\lambda}$, and it is found there that $[S^{\lambda}:D^{\lambda^{reg_e}}]=1$. The proof of this result is given for the representations of the symmetric group algebra only, but it is known to hold for Iwahori--Hecke algebras of type $A$. In 2019 Dimakis and Yue [DY] introduced the regularisation of a partition on two parameters $reg_{e,i}$, of which James's regularisation is a special case. We describe a stronger version here and generalise the algorithm given by Fayers in [F1] for the $e$-regularisation of a partition on the abacus. In the forthcoming paper [MB], the author shows that when the partition $\lambda^{reg_{e,i}}$ is $e$-regular, then $[S^{\lambda}:D^{\lambda^{reg_{e,i}}}]=1$. It is also worth noting the work developed in these notes has already been applied in [F2] to prove isomorphisms between crystals based on partitions, in turn isomorphic to the crystal $B(\Lambda_0)$ of the basic representation of $U_q(\mathfrak{\widehat{sl}}_e)$.
\vspace{2mm}

Most of the present notes, and all original results contained in it, are concerned with the combinatorics of the Young diagram and the abacus associated to an integer partition. We introduce these in section $2$. Next, we describe James's regularisation map, its generalisations on two parameters, and the interpretation of these maps in terms of decomposition numbers of Iwahori--Hecke algebras. Section $4$ contains, with proof, the algorithm on the $e$-abacus of a partition $\lambda$ yielding its generalised regularisation $\lambda^{reg_{e,i}}$ for each suitable $i$. We finalise this short paper giving structural results on $(e,i)$-regularisation classes in section 5.

\section{Background}
\subsection{Partitions and Young diagrams}
First, we fix some notation. We denote by $\N$ the set of all positive integer numbers $\{1,2,3,\ldots\}$ and let $n,e,i\in\N$, with $e\geq{2}$, $0<i<e$.

A \textit{partition} of an integer $n$ (written $\lambda\vdash{n}$) is  a weakly decreasing sequence $\lambda=(\lambda_1,\lambda_2,\ldots)$ of non-negative integers whose sum is $n=\abs{\lambda}=\lambda_1+\lambda_2+\cdots$. We usually group equal terms and omit zeros. For example, we write the partition $(3,3,3,1,1,0,0,\ldots)$ as $(3^3,1^2)$. A partition is \textit{$e$-regular} if it has no $e$ equal non-zero parts. That is, if there is no $j\in\N$ such that $\lambda_j=\cdots=\lambda_{j+e-1}>0$. Otherwise we say that $\lambda$ is \textit{$e$-singular}.

We usually represent partitions via their associated \textit{Young diagrams}. Given a partition $\lambda$, we define its Young diagram as the set
\begin{center}
$[\lambda]:=\{(a,b)\in\N^2:b\leq{\lambda_a}\}$.
\end{center}
A pair $(a,b)\in[\lambda]$ is called a \textit{node} of $\lambda$. We draw each node as a squared box and associate $(a,b)$ with its southeast corner. Given a certain Young diagram $[\lambda]$, its conjugate $[\lambda']$ is obtained by interchanging the rows and columns of $[\lambda]$. We denote by $\lambda':=(\lambda_1',\lambda_2',\ldots)$ the partition associated with $[\lambda']$.  We follow the English convention to represent Young diagrams (the lattice $\N^2$ has its $x$ axis pointing downwards and its $y$ axis increasing eastwards). Throughout these notes we abuse notation by denoting indistinctly the partition $\lambda$ and its associated Young diagram $[\lambda]$. The \textit{$e$-residue} of a node $(a,b)$ is the number $r\equiv b-a$ mod $e$.
\vspace{4mm}

\begin{exam} Let $\lambda=(8,5^2,4,2^3)\vdash{28}$. This partition is $e$-singular for $e=2,3$ and $e$-regular for $e\geq{4}$. The Young diagram of $\lambda$ is given below with the $3$-residue associated to each node:
\begin{center}
$\yngres(3,8,5^2,4,2^3)$
\end{center}
\end{exam}
\vspace{4mm}

Given a node $x=(c,d)\in\lambda$, we define the \textit{arm} of $x$, and write $arm_{c,d}$, as the set of nodes in $\lambda$ to the right of $x$ in the same row, and the \textit{leg} of $x$, $leg_{c,d}$, as the set of nodes in $\lambda$ lying below $x$ in the same column. The \textit{hook} of $(c,d)$ is the union of sets  $h_x=h_{c,d}=leg_{c,d}\,\cup \,arm_{c,d}\,\cup\, \{(c,d)\}$. The \textit{length} of either $h_{c,d}$, $leg_{c,d}$ or $arm_{c,d}$ is its cardinality as a set. The lowest node in $leg_{c,d}$ is called the \textit{foot} of $h_{c,d}$, and the rightmost node in $arm_{c,d}$ is called the \textit{hand} of $h_{c,d}$.\newline
The \textit{rim} of $[\lambda]$ is the set of nodes $(c,d)\in[\lambda]$ such that $(c+1,d+1)\not\in[\lambda]$. A \textit{skew-hook} of length $l$ is a set of $l$ consecutive nodes in the rim of $[\lambda]$ which can be removed such that the resulting diagram is the Young diagram of a partition. Lemma 18.4 in [J2] shows there is a one-to-one correspondence between the hooks and the skew-hooks of $[\lambda]$.\newline 

\subsection{The abacus}

Let $\lambda\vdash{n}$ and $s$ an integer greater than or equal to the number of non-zero entries of $\lambda$. For $1\leq{k}\leq{s}$ define $\beta_k:=\lambda_k+s-k$. The set $\{\beta_1,\ldots,\beta_s\}$ is called the $s$-\textit{beta-set} for $\lambda$.\newline
Given $e$, we define a $e$-\textit{abacus display} for $\lambda$ by taking $e$ vertical runners, numbered $0,\ldots,e-1$ from left to right, and labelling each position in runner $j$ with the integers $j,j+e,j+2e,\ldots$. For example, if $e=3$ we have the numbering:
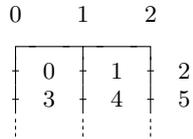
\begin{center}
\abacus(e{0}e{ }e{1}e{ }e{2}e{ },ldmdr,{ne{0}}ne{1}ne{2},ne{3}ne{4}ne{5},ve{ }ve{ }v)
\end{center}
Now, We place a bead at position $\beta_k$ for each $k$. We say that a position in the abacus is \textit{occupied} if there is a bead on it. Otherwise we say that the position is \textit{empty}.
\vspace{4mm}

\begin{exam} Let $\lambda=(9,7,6^2,4,3^2,2)$ with $e=5$ and $s=10$. Then $(\beta_1,\ldots,\beta_{10})=(18,15,13,\\12,9,7,6,4,1,0)$ and the $5$-abacus display associated to $\lambda$ is
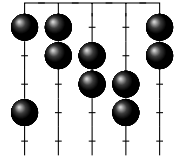
\begin{center}
\abacus(lmmmr,bbnnb,nbbnb,nnbbn,bnnbn,nnnnn)
\end{center}
\end{exam}

\begin{rem}
Note that the nodes corresponding to positions in the same runner have the same $e$-residue. If a node $(a,b)$ in the rim of $[\lambda]$ is associated to position $t$ in an abacus display for $\lambda$, then $t+e$ corresponds to the node $(a-k,b+(e-k))$ where $k$ is the number of beads between $t$ and $t+e$. Thus, $b+e-k-(a-k)=b-a+e\equiv b-a$ mod $e$.
\label{abres}
\end{rem}
\vspace{4mm}

We may abuse notation again and denote $\lambda$ indistinctly to allude either to the partition, its Young diagram, or an abacus display associated with the partition when it is clear which one of the three we are referring to.
\vspace{8mm}

\section{Regularisation maps}
\subsection{The $(e,i)$-regularisation of a partition}

The content of this section is, for the most part, non-standard. We define the $(e,i)$-regularisation of a partition, obtain James's regularisation as a particular instance, and deduce some properties of $(e,i)$-regular partitions.

For given natural numbers $l$ and $r<e$ with $ir+l\equiv 1$ $mod$ $e$, we define the \textit{$(l,r)$-th $(e,i)$-ladder} to be the subset of $\N^2$ defined as:
\begin{align}
\mathcal{L}_l^r=\{(a,b)\in\N\times\N  \,|\, l=(e-i)b+ai+1-e,\,  r\equiv b-a \,\, mod\,\, e\}
\label{laddesc}
\end{align} 
Note that if $ir+l\not\equiv 1$ $mod$ $e$, then no pair of natural numbers satifies the equations in (\ref{laddesc}), and thus $\mathcal{L}_l^r$ is empty.\newline
It is clear that for any $(a,b)\in\N^2$, there is a unique pair $(l,r)$ with $(a,b)\in\mathcal{L}_l^r$. We may refer to it simply as the $l$-th ladder if there is no doubt which $e$, $i$ and $r$ we are referring to. The $l$-th ladder of a partition $\lambda$ is the set $[\lambda]\cap\mathcal{L}_l^r$. In order to differentiate the nodes of $\mathcal{L}_l^r$ in $\lambda$ from those not in $\lambda$, we keep the name ``nodes" for those in $[\lambda]\cap\mathcal{L}_l^r$  and call the elements in $\mathcal{L}_l^r$ \textit{rungs}. We write $\mathcal{L}(a,b)$ to denote the ladder of the node $(a,b)$ in $\lambda$.
\vspace{4mm}

\begin{exam} Let $(e,i)=(6,4)$. We draw below the first $(6,4)$-ladders in a $\N^2$ box lattice indicating row and column, and labelling each box with the corresponding $(6,4)$-ladder:

\begin{center}
\newcommand\eleven{11}
\Yboxdim{30pt}
\gyoung(::1:2:3:4:5:6:7:8:9^1\hdts,:1<1,0><3,1><5,2><7,3><9,4><\eleven,5><13,0><15,1><17,2>^1\hdts,:2<5,5><7,0><9,1><\eleven,2><13,3><15,4><17,5><19,0><21,1>^1\hdts,:3<9,4><\eleven,5><13,0><15,1><17,2>^1\hdts,:4<13,3><15,4><17,5><19,0><21,1>^1\hdts,:5<17,2>^1\hdts,^1\vdts^1\vdts)
\end{center}
\end{exam}
\vspace{4mm}

\begin{lemma} Let $e,i\in\N$ with $1\leq i\leq e-1$ and $(a,b)\in\mathcal{L}_l^r$. Then $\mathcal{L}_l^r=\{(a-k(e-i),b+ki)\,\mid\,k\in\Z\}$.
\end{lemma}
\begin{proof} It is straightforward to check that every node of the form $(a-k(e-i),b+ki)$ for $k\in\Z$ lies in the same $(l,r)$-th $(e,i)$-ladder. That is, $\{(a-k(e-i),b+ki)\,\mid\,k\in\Z\}\subset\mathcal{L}_l^r$.

Assume $(a,b)$ and $(c,d)$ lie in the $(l,r)$-th $(e,i)$-ladder. We have then the equations:
\[
l=(e-i)b+ai+1-e,\; \qquad l=(e-i)d+ci+1-e.
\]
Subtracting the two equations we obtain $(e-i)b+ia=(e-i)d+ic$.  Take $h\in\Z$ such that $d=b+h$. Substituting in the identity we obtain $(c,d)=(a-\frac{(e-i)h}{i},b+h)$. Now, if $gcd(e,i)=1$ this implies that $i\mid h$, so there is a $k\in\Z$ such that $h=ki$, and we obtain $(c,d)=(a-k(e-i),b+ki)$ as we claimed. If $gcd(e,i)>1$, recall the residue $r\equiv b-a$ mod $ e$, so we have $r+me=b-a$ for some $m\in\Z$. Since $r$ is also the $e$-residue of $(c,d)$, we have $r+m'e=d-c=b-a+\frac{eh}{i}$ for $m'\in\Z$. But $\frac{eh}{i}$ must be a multiple of $e$, that is $\frac{eh}{i}=ke$, and so $h=ki$ for $k\in\Z$ and we obtain again $(c,d)=(a-k(e-i),b+ki)$.
\end{proof}

\begin{rems} 
Note that in the proof above we used the $e$-residue $r$ only for the case in which $e$ and $i$ are not coprime. When $gcd(e,i)=1$, the first of the equations in (\ref{laddesc}) determines the $e$-residue. We can then ``drop" $r$ and write
\[
\mathcal{L}_l:=\{(a,b)\in\N\times\N  \,|\, l=(e-i)b+ai+1-e\}=\{(a-k(e-i),b+ki)\,\mid\,k\in\Z\}.
\]
From Remark~\ref{abres}, if we wish to show two nodes $x,y$ in the rim of $[\lambda]$ corresponding to positions lying in the same runner of an abacus display for $\lambda$ are in the $(l,r)$-th ladder, it is enough to check that $x$ and $y$ define the same $l$. We will use this fact in the proof of Proposition~\ref{main}.
\end{rems}
\vspace{4mm}

\begin{lemma} If $(a,b),(c,d)\in\N^2$ lie in the same $(ke,ki)$-ladder, then they lie in the same $(e,i)$-ladder.
\label{samelad}
\end{lemma}

\begin{proof} For any two nodes $(a,b),(c,d)\in\N^2$ lying in the same $(ke,ki)$-ladder, we obtain the identity $(ke-ki)b+kia=(ke-ki)d+kic$ which clearly implies $(e-i)b+ia=(e-i)d+ic$. Similarly, $(a,b)$ and $(c,d)$ must have the same $e$-residue, since $b-a\equiv d-c$ $mod$ $ke$ implies $b-a\equiv d-c$ $mod$ $e$. Therefore $(a,b)$ and $(c,d)$ lie in the same $(e,i)$-ladder.
\end{proof}
\vspace{4mm}

\begin{exam} Let $\lambda=(5,3,1^4)$. This partition is $e$-singular for $e=2,3,4$ and $e$-regular for $e\geq{5}$. For $e=3$ and $i=1,2$ respectively, we represent in each box of its Young diagram the ladder associated to each node:
\begin{center}
\begin{tabular}{cccccc}
$(e,i)=(3,1)$& & & &$(e,i)=(3,2)$&\\
&\young(13579,246,3,4,5,6)& & & & \young(12345,345,5,7,9,<11>)\\
\end{tabular}
\end{center}
\end{exam}
\vspace{4mm}




Now we can define the $(e,i)$-regularisation of a partition:

\begin{definition}[\textbf{$(e,i)$-regularisation}] Let $\lambda\vdash{n}$. Then its $(e,i)$-regularisation is the partition $\lambda^{reg_{e,i}}\vdash{n}$ whose Young diagram is built from that of $\lambda$ by moving all nodes upwards in their $(e,i)$-ladders as high as possible subject to the resulting configuration being the Young diagram of a partition. If no node can be moved, then $\lambda^{reg_{e,i}}:=\lambda$.\newline
\end{definition}

It is not clear from the definition above that $(e,i)$-regularisations are well defined, which will be the matter of sections $4$ and $5$. It is worth noting that this definition of the $(e,i)$-regularisation is a stronger version of the one given by Dimakis and Yue in [DY], who consider also those $[\lambda^{reg_{e,i}}]$ not defining an integer partition.
\vspace{4mm}

\begin{definition}[\textbf{$(e,i)$-regular partitions}]
We say that a partition $\mu\vdash{n}$ is \textit{$(e,i)$-regular} when $\mu^{reg_{e,i}}=\mu$. Otherwise we say $\mu$ is \textit{$(e,i)$-singular}. 
\end{definition}

From the two previous definitions it is clear that $\lambda,\mu\vdash{n}$ have the same $(e,i)$-regularisation if and only if $[\lambda]\cap\mathcal{L}_l^r=[\mu]\cap\mathcal{L}_l^r$ for all $l$, and $0\leq r\leq e-1$. In this case, we say that $\mu\vdash{n}$ belongs to the \textit{$(e,i)$-regularisation class of $\lambda$}, denoted $\mu\in\mathfrak{R}_{(e,i)}(\lambda)$. It is easy to check that the relation between two partitions of the same integer defined by having the same $(e,i)$-regularisation is indeed an equivalence relation. As we prove in Section $5$, we can take $(e,i)$-regular partitions as representatives of their $(e,i)$-regularisation class.

\begin{rems}
The regularisation map defined by James in [J1] is our $(e,1)$-regularisation. Thus, a partition is $e$-regular if and only if it is $(e,1)$-regular.

Note that Lemma~\ref{samelad} implies that if $\mu\in\mathfrak{R}_{(ke,ki)}(\lambda)$, then $\mu\in\mathfrak{R}_{(e,i)}(\lambda)$.
\end{rems}

\vspace{4mm}

Next, we see examples of the $(e,i)$-regularisation process:

\begin{exam} Let $\lambda=(5,4^2,1^3)\vdash{16}$ and $e=4$. This partition is $4$-regular so $\lambda^{reg_{4,1}}=\lambda$. For $i=2,3$ we have $\lambda^{reg_{4,2}}=(6,5,3,1^2)$ and $\lambda^{reg_{4,3}}=(10,2,1^4)$. We can see this from the Young diagrams below, where the nodes moved in the $(4,i)$-regularisation are labelled by different letters according to the $(4,i)$-ladder they belong to. Note that $\lambda^{reg_{4,3}}$ is $4$-singular despite $\lambda$ being $4$-regular. 
\vspace{4mm}

The case $(e,i)=(4,2)$
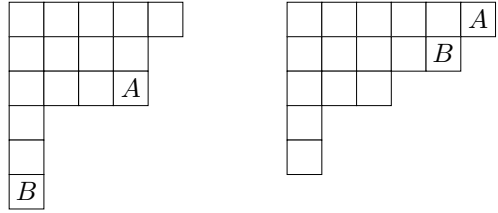
\begin{center}
\begin{tikzpicture}[scale=1]
\tgyoung(0cm,0cm,;;;;;:::;;;;;;<A>,;;;;::::;;;;;<B>,;;;;<A>::::;;;,;:::::::;,;:::::::;,;<B>)
\end{tikzpicture}
\end{center}

The case $(e,i)=(4,3)$
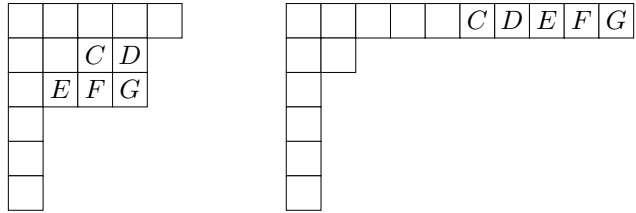
\begin{center}
\begin{tikzpicture}[scale=1]
\tgyoung(0cm,0cm,;;;;;:::;;;;;;<C>;<D>;<E>;<F>;<G>,;;;<C>;<D>::::;;,;;<E>;<F>;<G>::::;,;:::::::;,;:::::::;,;:::::::;)
\end{tikzpicture}
\end{center}

\end{exam}
\vspace{4mm}

Recall the definition of the hook $h_x$ for a node $x\in\lambda$ in Section $2$. The following definition will be useful later on to characterise $(e,i)$-regular partitions:

\begin{definition}[\textbf{$(e,i)$-hooks.}] Let $\lambda\vdash{n}$ and $x\in\lambda$ such that $e$ divides $|h_x|$ and $\frac{|h_x|}{|leg_x|}=\frac{e}{e-i}$. We call such hooks $h_x$ $(e,i)$-hooks.
\end{definition}
\vspace{4mm}

\begin{lemma} Let $\mu,\lambda\vdash{n}$ with $\lambda\in\mathfrak{R}_{(e,i)}(\mu)$ and $\lambda\neq\mu$. Then either $\mu$ or $\lambda$ has a $(e,i)$-hook.
\label{uniqueness}
\end{lemma}

\begin{proof}
Let $b\in\N$ be minimum such that $\lambda_b'\neq\mu_b'$ and lets assume without loss of generality that $\lambda_b'<\mu_b'$. We will show that $\mu$ has an $(e,i)$-hook. Let $(a,b)\in\mu$ be the node at the bottom of $\mu_b'$. Since $\lambda\in\mathfrak{R}_{(e,i)}(\mu)$ and $(a,b)\not\in\lambda$, there needs to be a node in $\lambda$ in the $(e,i)$-ladder of $(a,b)$ above row $a$ which is not in $\mu$. For a given node $x\in\lambda$ ,let $\mathcal{L}_{<a}^{\lambda}(x)$ denote the subset of nodes of $\lambda$ in the $(e,i)$-ladder of $x$ above row $a$. Then, because every node in the ladder of $(a,b)$ has a node to its left in the ladder of $(a,b-1)$, we have $|\mathcal{L}_{<a}^{\lambda}((a,b-1))|\geq|\mathcal{L}_{<a}^{\lambda}((a,b))|$. Therefore, $|\mathcal{L}_{<a}^{\mu}((a,b-1))|>|\mathcal{L}_{<a}^{\mu}((a,b))|$. Thus, there exist a node in the ladder of $(a,b-1)$ in $\mu$, namely $(a-k(e-i),b+ki-1)$ for a certain $k>0$ such that $(a-k(e-i),b+ki)\not\in\mu$. But in that case we have a $(e,i)$-hook in $\mu$ with foot $(a,b)$ and hand $(a-k(e-i),b+ki-1)$.

\end{proof}

Next, we give a characterisation of the $(e,i)$-hooks in a partition $\lambda$ via its abacus display.

\begin{lemma} A partition $\lambda$ contains an $(e,i)$-hook if and only if there is a bead $b$ in the abacus display for $\lambda$ such that there is an empty position $b-ke$, $k\in\N$ in the runner of $b$ with $k(p-i)$ beads  between $b$ and $b-ke$ (excluding $b$) reading horizontally from $b-ke$ to $b$.
\label{abacus1}
\end{lemma}
\begin{proof} First, assume $\lambda\vdash{n}$ contains an $(e,i)$-hook and take an abacus display for $\lambda$ with $e$ runners and $r$ beads. By the definition of $(e,i)$-hook, there is a node $(c,d)\in\lambda$ such that the hook associated to it, $h_{c,d}$ satisfies $\frac{|h_{c,d}|}{|leg_{c,d}|}=\frac{e}{e-i}$. This means there is a $k\in\N$  with $k(e-i)$ being the leg length and $ki-1$ being the arm length of $h_{c,d}$. Consider the bead $b$ in the abacus display of $\lambda$ associated to row $c$. The correspondence between beads and rows, and empty positions and columns, tell us that there will be  $k(e-i)$ beads (and therefore $ki$ empty spaces) in the next $ke$ positions to the left of the bead $b$. Since $|leg_{c,d}|=k(e-i)$, in row $c+k(e-i)+1$ there must be at least one node less than in row $c+k(e-i)$. This means there will be an empty position at $b-ke$ in the same runner as $b$.

Now assume we have a bead $b$ in the abacus display for $\lambda$ as in the formulation. Let $b$ be the $(r-c+1)$-th bead in such abacus display. Using again the correspondence between beads and rows and empty positions and columns, together with the one between hooks and skew hooks in the Young diagram of a partition (see p. 73 in [J2]), it is straightforward that the hook in $\lambda$ associated with the section of the abacus between $b$ and $b-ke$ will be such that $|h_{c,d}|=ke$ and $|leg_{c,d}|=k(e-i)$.
\end{proof}

We can define a total order on partitions of $n$ by letting $\lambda>\mu$ if for some $j$ we have $\lambda_j>\mu_j$ and $\lambda_k=\mu_k$ for all $k<j$. This is called the \textit{lexicographic order}. We define a partial order on partitions we will need next:

\begin{definition}[\textbf{Dominance order}] Let $\lambda,\mu\vdash{n}$. We say that $\lambda$ is more dominant than $\mu$, and write $\lambda\unrhd\mu$, if for all $j\geq 1$ we have $\sum_{k=1}^j\lambda_k\geq \sum_{k=1}^j\mu_k$. We write $\lambda\triangleright\mu$ to indicate $\lambda\unrhd\mu$, $\lambda\neq\mu$.
\end{definition}
\vspace{4mm}

We end this section with the following result on the $(e,i)$-regularisation classes:

\begin{lemma} Let $\lambda\vdash{n}$ and $\mu\in\mathfrak{R}_{(e,i)}(\lambda)$ built from $\lambda$ by moving some nodes up in their ladders. Then $\mu\triangleright\lambda$.
\label{dom}
\end{lemma}
\begin{proof} Let $\mathfrak{b}_{\mu}(\lambda)$ denote the set of integers $k$ such that $\mu_k>\lambda_k$, and $\mathfrak{b}_{\lambda}(\mu)$ the set of integers $m$ such that $\lambda_m>\mu_m$. From the construction we have that $\mathfrak{b}_{\mu}(\lambda)$ defines the rows of $\lambda$ to which nodes are added to form $\mu$, and $\mathfrak{b}_{\lambda}(\mu)$ the rows of $\lambda$ where nodes are moved up in their runners. Therefore, for each $k\in\mathfrak{b}_{\mu}(\lambda)$ there is at least one $m\in\mathfrak{b}_{\lambda}(\mu)$ such that $k<m$. This, together with the fact that both $\lambda$ and $\mu$ have the same number of nodes, guarantees $\mu\triangleright\lambda$.
\end{proof}
\vspace{6mm}

\subsection{$(e,i)$-regularisations and Hecke algebras of type $A$}

The $(e,i)$-regularisation is defined in the context of the representation theory of the symmetric group $\mathfrak{S}_n$, where the partitions of $n$ determine the irreducible modules in characteristic zero $S^{\lambda}$, called \textit{Specht modules}. Their restrictions to fields of positive characteristic $p$ are not irreducible, but all irreducible modules $D^{\mu}$ can be realised as heads of particular quotients of Specht modules for $\mu$ $p$-regular. The decomposition numbers $d_{\lambda\mu}=[S^{\lambda}:D^{\mu}]\geq{0}$ are the multiplicities of the irreducible modules $D^{\mu}$ as composition factors of each Specht module $S^{\lambda}$. These concepts extend to the Iwahori--Hecke algebras $\He_{\F,q}(\mathfrak{S}_n)$, where the role played by the prime $p$ in representations of $\mathfrak{S}_n$ is played here by the ``quantum characteristic" $e\in\N$. The $(e,1)$-regularisation agrees with the regularisation map defined by James in [J1], and always returns an $e$-regular partition.  In that landmark paper, James showed the following result for symmetric groups (so substituting $e$ by $p=char(\mathbb{F})$). Since it is known to hold for Iwahori--Hecke algebras, we will write it here directly in terms of $e$:

\begin{thm}[\textbf{J1, Theorem A}] Let $\lambda\vdash{n}$. Then, $[S^{\lambda}:D^{\lambda^{reg_{e,1}}}]=1$.
\end{thm}

The motivation of this paper is to develop the combinatorial tools to prove the following generalisation of James's result above, which is the focus of [MB]:

\begin{thm} Let $\lambda\vdash{n}$. For all $0<i<e-1$ we have that if $\lambda^{reg_{i,e}}$ is $e$-regular, then $[S^{\lambda}:D^{\lambda^{reg_{e,i}}}]=1$.
\end{thm}

We will not go further in the representation theory here, and refer the interested reader to the classical books by James [J2] on representations of $\mathfrak{S}_n$, and by Mathas [Ma] on Iwahori--Hecke algebras. 
\vspace{8mm}

\section{An algorithm for the $(e,i)$-regularisation on the abacus}

The purpose of this section is to prove the following result, which will allow us to show that the $(e,i)$-regularisation is well-defined, while at the same time giving an algorithm for constructing the $(e,i)$-regularisation of a partition.

\begin{prop} If $\lambda$ has an $(e,i)$-hook, then there exists $\mu\in\mathfrak{R}_{(e,i)}(\lambda)$ with $\mu\unrhd\lambda$.
\label{moredom}
\end{prop}

\begin{proof} Let $k$ be maximal such that $\lambda$ has an $(e,i)$-hook at $x$ with $|h_x|=ke$ and $|leg_x|=k(e-i)$ and assume there exists $\mu\in\mathfrak{R}_{(e,i)}(\lambda)$ with $\mu\unrhd\lambda$.  If $k>1$, replacing $(e,i)$ with $(ke,ki)$ we could find a partition $\mu$ in the $(ke,ki)$-regularisation class of $\lambda$ which is more dominant, but by Lemma~\ref{samelad} this implies that $\mu\in\mathfrak{R}_{(e,i)(\lambda)}$. Thus, we can safely assume $k=1$ for the rest of the proof.

In order to facilitate its digestion, we have divided the proof in the remaining statements given in the present section, namely Lemma~\ref{phipsi}, Lemma~\ref{ladders} and Proposition~\ref{main}. Given $\lambda$ containing a $(e,i)$-hook, our algorithm will yield a partition $\phi(\lambda)\in\mathfrak{R}_{(e,i)}(\lambda)$ with $\phi(\lambda)\triangleright\lambda$.
\end{proof}


Let $\lambda\vdash{n}$ and assume we are in the situation of Lemma~\ref{abacus1}, that is, there is 
an empty position $s_1$ such that $b_1:=s_1+e$ has a bead on it and with exactly $i-1$ other empty positions $s_1<s_2<\cdots<s_i<b_1$ in runners labelled $x_1,\ldots,x_i$, respectively. Lets take $s_1$ maximal with that property. Numerate the beads after $b_1$ in the runners $x_j$, $j=1\ldots,i$ consecutively $b_2<\cdots<b_l$.\newline
It is important to stress that the subscript of the beads is not related with the runner it lies in, so that $b_{3}$ might be in runner $x_1$, for example. However, $s_m$ lies in $x_m$ for $m\leq{i}$.

Let $t_1<t_2<\cdots$ be the empty positions after $s_i$ not on the runners $x_j$. Let $c\in\{1,\ldots,l\}$ be minimal such that $t_c<b_{c+1}$. If no such condition is ever satisfied, then set $c=l$.\newline
Now we construct a new abacus, with associated partition $\phi(\lambda)$, by moving a bead a position up its runner from $b_k$ and a bead one place down the runner to $t_k$ for $k=1\ldots,c$. For $i>1$, it may be the case that not all $b_k-e$ are empty, nor all $t_k-e$ are occupied, but from the construction of $\phi(\lambda)$ if $b_k-e$ has a bead on it then $b_k-e=b_k'$ with $k'<k$, and so $b_k'$ will be moved to $b_k'-e=b_k-2e$, and this process must end since $b_k$ is in a runner $x_j$ and $s_j$ is empty in that runner. The case for the positions $t_k$ is similar. 

\begin{center}
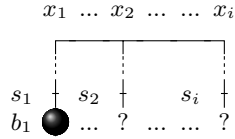

\abacus(e{ }e{x_1}e{...}e{x_2}e{...}e{...}e{x_i},e{ }ldmddr,e{ }ve{ }ve{ }e{ }v,e{s_1}ne{s_2}ne{ }e{s_i}n,e{b_1}be{...}e{?}e{...}e{...}e{?})
\captionof{figure}{abacus for $\lambda$}
\end{center}

Since $t_c>b_c$, it is clear that $\phi(\lambda)>\lambda$ in the lexicographic order of partitions. By virtue of Lemma~\ref{dom} we only need to show that $\phi(\lambda)\in\mathfrak{R}_{(e,i)}(\lambda)$ in order to prove our algorithm works.

\begin{lemma} Let $\lambda\vdash{n}$ with $c$ as above. Then $c$ is also minimal with $t_{c+1}<b_{c+1}$.
\label{phipsi}
\end{lemma}
\begin{proof}
Consider $s_1$ the maximal position in the abacus display of $\lambda$ giving rise to a $(e,i)$-hook in $[\lambda]$. Let $c$ be minimal such that $t_c<b_{c+1}$ and let $d$ be minimal such that $t_{d+1}<b_{d+1}$. We show $c=d$. Observe that $d\geq{c}$, since $b_{c+1}>t_c>t_{c-1}>b_c$ and $d<c$ would contradict the choice of $c$. Lets assume then $d>c$. In particular $l\geq{c+1}$ and $t_c<b_{c+1}<t_{c+1}$. We will show that this last inequality leads to a contradiction in the maximality of $s_1$. Write $b_{c+1}=s_1+ue+v$, $0\leq{v}<e$ lying in runner $x_j$. Position $s_1+v$, also in $x_j$ must lie empty. Now, there are $c$ beads between $b_1$ and $b_{c+1}$ in runners $x_v$, $v\leq{i}$. To count the number of beads in the other runners between $s_1+v$ and $b_{c+1}$, observe there are $(e-i)-v-1+j$ beads between $s_1+v$ and $b_1$, and $(e-i)(u-1)+v-j+1-c$ beads between $b_1$ and $b_{c+1}$ after subtracting all empty positions $t_1,\ldots,t_c$ lying in those runners. Therefore we have in total $u(e-i)$ beads between $s_1+v=b_{c+1}-ue$ and $b_{c+1}$. Since $s_1<s_1+v$, we have a contradiction with the choice of $s_1$. Thus, necessarily we have $t_c<t_{c+1}<b_{c+1}$, and therefore $c=d$.
\end{proof}
\vspace{4mm}

\begin{exam} Let $(e,i)=(7,4)$ and $\lambda=(14,13,5,4^5,3,1^2)$ with abacus display
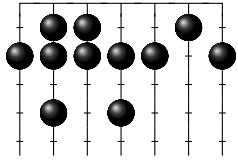
\begin{center}
\abacus(lmmmmmr,nbbnnbn,bbbbbnb,nnnnnnn,nbnbnnn,nnnnnnn)
\end{center}
where $r=11$, $(b_1,b_2,b_3,b_4,b_5)=(7,10,11,13,24)$ and $(t_l)=(12,15,16,19,23...)$, so $c=4$ in the definition of $\phi(\lambda)$ with abacus display
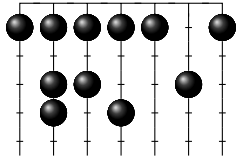
\begin{center}
\abacus(lmmmmmr,bbbbbnb,nnnnnnn,nbbnnbn,nbnbnnn,nnnnnnn)
\end{center}
representing the partition $(14,13,11,9^2,1)$.
\end{exam}
\vspace{4mm}


\begin{lemma} Let $\lambda\vdash{n}$ and consider an abacus display for $\lambda$ with $r$ beads. Suppose that $\xi$ is obtained from $\lambda$ by moving a bead from position $s_1+e$ to $s_1$. Define $s_j$, $j\leq{i}$ as above and let $k_j$ be the number of beads between $s_j$ and $s_{j+1}$. Then $k_j=s_{j+1}-s_j-1$ for $j\neq{i}$ and $k_i=s_1+e-s_i-1$. We can express then $s_j$ in terms of $s_1$ and the $k_l$'s with $l<j$ as $s_j=s_1+j-1+\sum_{l=1}^{j-1}k_l$. \newline
We would like to find the ladders to which the nodes lying in $\lambda\setminus\xi$ belong. Note that no two of the $e$ nodes in $\lambda\setminus\xi$ has the same residue. Let $m$ be the number of empty spaces before position $s_1$ in either display. From equation (\ref{laddesc}) we have
$$
l_{s_1}=em+(r-s_1-1)i+1
$$
is the number determining the $(e,i)$-ladder of the node defined by $s_1$. Consider $j\in\{1,\ldots,e-1\}$ and let $l_{s_j}$ be the ladder associated to position $s_j$ in the abacus. It is possible to obtain the $j$-th $(e,i)$-ladder from the $(j-1)$-th $(e,i)$-ladder by adding $e-i$ if position $s_1+j$ is empty and subtracting $i$ otherwise. That is, the ladders will be precisely those numbered 
\begin{align}
&l_{s_1},\quad l_{s_1}-i,\ldots,l_{s_1}-k_1i,\notag\\ 
&l_{s_2},\quad l_{s_2}-i,\ldots,l_{s_2}-k_2i, \notag\\
&\vdots \label{setlad}\\
&l_{s_i},\quad l_{s_i}-i,\ldots,l_{s_i}-k_ii.\notag
\end{align}
where we can see $l_{s_j}=e(m+j-1)+(r-s_j-1)i+1$ for $j\leq{i}$. Note that $l_{s_i}-k_ii=e(m-1)+(r-s_1)i+1$.
\label{ladders}
\end{lemma}
\begin{proof}
To obtain $l_{s_1}$ we need to compute the row and column in $[\lambda]$ corresponding to $s_1$ in its abacus display. Position $s_1$ is the $(m+1)$-th empty space in the abacus, so it corresponds to column $m+1$ in the Young diagram. Moreover, the number of beads before $s_1$ is $s_1-m$. Thus, the row we are looking for is $r+m-s_1$, and $l_{s_1}$ follows. The rest is straight forward following the definition of $s_j$ in terms of $s_1$ and $k_l$ in the formulation, and the fact that  $l_{s_j}=l_{s_{j-1}}-(k_{j-1}+1)i+e$. For the expression of $l_{s_i}$ observe
$$\sum_{l=1}^{i}k_l=s_2-s_1-1+s_3-s_2-1+\cdots+s_i-s_{i-1}-1+s_1+e-s_i-1=e-i.$$ 
We use such identity to obtain the expression for $l_{s_i}-k_ii=e(m+i-1)+(r-s_i-k_i-1)i+1=e(m+i-1)+(r-s_1-i-\sum_{l=1}^ik_l)i+1$.
\end{proof}

\begin{rem} Observe that we can encode all information about the ladders in Lemma~\ref{ladders} by computing $l_{s_1}$ and the ordered tuple of $-i$ and $+e-i$ given by the consecutive $e-1$ entries in the abacus between $s_1$ and $b_1=s_1+e$.
\label{rema}
\end{rem}
\vspace{4mm}

\begin{prop} Let $\lambda\vdash{n}$. Then $\phi(\lambda)\in\mathfrak{R}_{(e,i)}(\lambda)$.
\label{main}
\end{prop}
\begin{proof} We proceed by induction on $n$, and for fixed $n$ by induction on lexicographic order. Assuming $c>1$, we build $\xi$  as the partition obtained from the abacus display of $\lambda$ by moving a bead from position $b_1=s_1+e$ to $s_1$. We can see that the bead $b_2$ defines a $(e,i)$-hook in $\xi$: since $c>1$, $t_1>b_2$ and every space between $s_2=b_2-e$ and $b_2$ not in one of the runners $x_z$ has a bead on it, and every space on runners $x_z$ must lie empty. Moreover, by the choice of $s_1$, $s_2$ is maximal with that property in $\xi$. For the construction of $\phi(\xi)$, we need $c'$ minimal such that $t_{c'}<b_{c'+2}$. Note that by Lemma~\ref{phipsi}, this last condition is equivalent to $t_{c'+1}<b_{c'+2}$. From the construction $\xi$ and $\lambda$ are equal after $b_1$, so $c'=c-1$, and we form $\phi(\xi)$ by moving all beads $b_2,\ldots,b_c$ upwards and beads from positions $t_1-e,\ldots,t_{c-1}-e$ downwards in their runners. Thus, one obtains $\phi(\lambda)$ from $\phi(\xi)$ by moving a bead down in its runner from $t_c-e$ to $t_c$. The construction of $\phi(\xi)$ ensures that $t_c-e$ is occupied. Let $j\leq{i}$ be such that $t_c$ lies between runners $x_j$ and $x_{j+1}$  (we set $x_{j+1}:=x_1$ if $j=i$).

\begin{center}
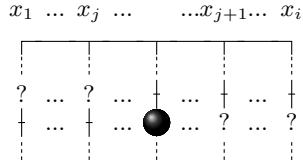

\abacus(e{x_1}e{...}e{x_j}e{...}e{ }e{...}e{x_{j+1}}e{...}e{x_i},ldmdmdmdr,ve{ }ve{ }ve{ }ve{ }v,e{?}e{...}e{?}e{...}ne{...}ne{...}n,ne{...}ne{...}be{...}e{?}e{...}e{?},ve{ }ve{ }ve{ }ve{ }v)
\captionof{figure}{abacus for $\phi(\lambda)$ showing a bead at position $t_c$}
\end{center}

From Lemma~\ref{ladders}, if $m$ is the number of empty spaces before $s_1$, the $(e,i)$-ladders removed to obtain $\xi$ from $\lambda$ are those given by equations~(\ref{setlad}): $l_{s_1},\ldots,l_{s_i}-k_ii$. \newline
We want to show the nodes in $\phi(\lambda)\backslash\phi(\xi)$ lie in the same $(e,i)$-ladders as those in $\lambda\backslash\xi$. We compute the number of empty spaces before $t_c-e:=s_1+ue+v$, $0<v<e$ in $\phi(\lambda)$. This is $ui+j-c$ positions in the runners $x_l$ minus the $c$ beads $b_1,\ldots,b_c$ added to the $c-1$ $t_l$'s. Thus, in total we have $m+ui+j-1$ empty spaces. By the construction of $\phi(\xi)$ we know that position $s_1+(u+1)e=b_l$ for a certain $l<c$ must lie empty in $\phi(\lambda)$ and we know that the spaces in runners $x_v$ lie empty and all the other will be occupied. This means that the sequence of $+(e-i)'s$ and $-i's$ mentioned in Remark~\ref{rema} will be the same subject to a cyclic shift (depending on $v$) replacing the $-i$ in the $v$-th element in that sequence for the rim $\lambda\backslash\xi$ by a $+e-i$ in the rim $\phi(\lambda)\backslash\phi(\xi)$. That is, we will have the same ordered tuple with a shift by $v$ and one of the $-i$'s (the one corresponding to $x_1$ in the tuple for $\phi(\lambda)$ substituted by a $+e-i$).\newline
Therefore, in the abacus display of $\phi(\lambda)$ there are $(\sum_{l=1}^{j}k_l)+j-v-1$ consecutive beads between the runner of $t_c$ and $x_{j+1}$. Then we have the same sequence of beads and spaces between $x_{j+1}$ and $x_j$ reading horizontally and left to right as there is in those runners between $s_1$ and $b_1$. After $x_{j+1}$ we have $v-j-\sum_{l=1}^{j-1}k_l$ beads. Using the information from the previous paragraph, the ladder associated with $t_c-e$ is $l_{t_c-e}=e(m+ui+j-1)+(r-s_1-ue-v-1)i+1=e(m+j-1)+(r-s_1-v-1)i+1$. Now, observe this is precisely the ladder associated with position $s_1+v$ in $\lambda$, and since both $t_c-e$ and $s_1+v$ belong to the same runner, they have the same $e$-residue. At runner $x_{j+1}$ in $\phi(\lambda)$ we have the ladder $l_{t_c-e}-(j-v-1+\sum_{l=1}^jk_l)i+(e-i)=e(m+j)+(r-s_1-j-\sum_{l=1}^jk_l-1)i+1=e(m+j)+(r-s_{j+1}-1)i+1=l_{s_{j+1}}$, which again is the same ladder as the one corresponding $s_{j+1}$ in $\lambda$ and therefore have the same $e$-residue. Thus, the ladders for the nodes in $\phi(\lambda)\backslash\phi(\mu)$ are:
\begin{align}
& l_{t_c-e},\quad l_{t_c-e}-i,\ldots,l_{t_c-e}-(j-v-1+\sum_{l=1}^jk_l)i,\notag\\
& l_{s_{j+1}},\quad l_{s_{j+1}}-i,\ldots,l_{s_{j+1}}-k_{j+1}i,\notag\\
&\qquad \vdots \notag\\
& l_{s_i},\quad l_{s_i}-i,\ldots,l_{s_i}-k_ii,\notag\\
& l_{s_1},\quad l_{s_1}-i,\ldots,l_{s_1}-k_1i,\notag\\
&\qquad \vdots \notag\\
& l_{s_j},\quad l_{s_j}-i,\ldots l_{s_j}-(v-j-\sum_{l=1}^{j-1}k_l)i\notag \\
\notag
\end{align}
Now, rearranging by placing the ladder $l_{s_1}$ first and counting all consecutive ladder appearing, and placing the first ladder after the last one (that is, $l_{s_j}-(v-j-\sum_{l=1}^{j-1}k_l)i=e(m+j-1)+(r-s_1-v)i+1=l_{t_c}-i$ followed by $l_{t_c}$), it is clear that the nodes removed to go from $\lambda$ to $\xi$ lie in the same $(e,i)$-ladders than those removed going from $\phi(\lambda)$ to $\phi(\xi)$. Since we also have that the nodes with the same ladder numbering in either $\lambda$ and $\phi(\lambda)$, lie in the same runner of their corresponding abacus displays, they also have the same $e$-residue.\newline
If $c=1$ then we replace  $\phi(\xi)$ by $\xi$ in the argument above.
\end{proof}

\vspace{8mm}

\section{Properties of $(e,i)$-regular partitions}

In this last section we use the work in sections 3 and 4 to determine that there exist a unique most dominant partition in each regularisation class and characterise $(e,i)$-regular partitions in terms of their Young diagram and abacus display.

\begin{cor} Let $\lambda\vdash{n}$. Then $\lambda^{reg_{e,i}}$ is the most dominant partition in its $(e,i)$-regularisation class. 
\label{onlylocked}
\end{cor}
\begin{proof} We have from Lemma~\ref{uniqueness} that in each $(e,i)$-regularisation class there is at most a single partition with no $(e,i)$-hooks and the algorithm proved in Section 4 shows that if we have a partition $\lambda$ in the class with an $(e,i)$-hook, we can build another partition $\phi(\lambda)$ in the class which is more dominant. It follows that $\lambda^{reg_{e,i}}$ is the unique $(e,i)$-regular partition in the class and that it must dominate every other partition contained in the the $(e,i)$-regularisation class.



\end{proof}
\vspace{4mm}

We can now prove  an important characterisation of $(e,i)$-regular partitions:

\begin{prop} A partition $\lambda$ is $(e,i)$-regular if and only if it does not have $(e,i)$-hooks.
\label{charreg}
\end{prop}

\begin{proof}
First we show that if $\lambda$ does not have a $(e,i)$-hook, then it is $(e,i)$-regular and we will proceed by contradiction.  Assume $\lambda$ is $(e,i)$-singular, so there is at least one node which can be moved up in its $(e,i)$-ladder such that the resulting configuration is the Young diagram of a partition. Consider the lowest and leftmost such node $x=(a,b)\in\lambda$. In particular, we have $(a+1,b)\not\in\lambda$ and, if $b>1$, $y=(a,b-1)\in\lambda$ can not be moved up in it's $(e,i)$-ladder.  Let $\mathcal{L}_{<a}(x)$ and $\mathcal{L}_{<a}(y)$ be, respectively, the subsets of the ladders of $x$ and $y$ in $\lambda$ lying in a row $<a$.\newline
Observe that we have $|\mathcal{L}_{<a}(x)|\leq |\mathcal{L}_{<a}(y)|$. Now, if $|\mathcal{L}_{<a}(x)|< |\mathcal{L}_{<a}(y)|$, let $(a-k(e-i),b+ki-1)\in\lambda$ with $(a-k(e-i),b+ki)\not\in\lambda$. Then, there is an $(e,i)$-hook with foot $(a,b)$ and hand $(a-k(e-i),b+ki-1)$. Thus, $|\mathcal{L}_{<a}(x)|=|\mathcal{L}_{<a}(y)|$, but this means that if $(a,b)$, $b>1$ moves up, also $(a,b-1)$ moves up, which contradicts our choice of $x=(a,b)$. If $b=1$, this implies that $(a,1)$ can not be moved up and contradicts the choice of $(a,1)$.

\vspace{4mm}

Now, assume there is a $(e,i)$-hook at node $(a,b)\in\lambda$. Then the lowest node in column $b$ is unlocked, since it has an empty rung in its ladder above $(a,b)$ such that the rung immediately to its left is occupied. 
Thus, by Corollary~\ref{onlylocked}, $\lambda$ is $(e,i)$-singular.


\end{proof}

\begin{rem} Note that from Proposition~\ref{charreg} and Lemma~\ref{abacus1} we obtain a characterisation of a $(e,i)$-regular partition via its abacus display.

\end{rem}

\vspace{4mm}

Using the previous lemma, we can obtain a necessary condition for a partition to be $(e,i)$-regular:

\begin{lemma} If $\lambda\vdash{n}$ is $(e,i)$-regular, then there are no rungs $A$, $B$ and $C$ in the same $(e,i)$-ladder with $A$ in a row above $B$, and $B$ above $C$, such that $B\in\lambda$ and $A,C\not\in\lambda$.
\label{ABC}
\end{lemma}
\begin{proof} If $i=1$ this is straightforward. For $i>1$, let $\lambda$ be a partition with such three rungs in the same $(e,i)$-ladder and assume also that $\lambda$ is $(e,i)$-regular. Then, in the diagram of $\lambda$ there are no $(e,i)$-hooks. Denote the rungs by $B=(a,b)$, $A=(a-k_A(e-i),b+k_Ai)$ and $C=(a+k_C(e-i),b-k_Ci)$ with $k_A,k_C>0$. Among all the empty rungs above and below $B$ in its $(e,i)$-ladder let $A$ and $C$ be those closer to $B$. That is, $k_A$ and $k_C$ are minimal such that $A$, $C\not\in{\lambda}$.\newline
Let $1\leq{h_1}\leq{e-i}$ be minimum such that $(a+k_C(e-i)-h_1,b-k_Ci)\in\lambda$. By our hypothesis, for all $k<k_A$, $(a-k(e-i)-h_1,b+ki)\in\lambda$ and for all $k\leq{k_C}$, $(a+k(e-i)-h_1,b-ki)\in\lambda$. Now let $h_1\leq{h_2}\leq{e-i}$ minimum such that $(a+k_C(e-i)-h_2,b-k_Ci+1)\in\lambda$. Again, since by our hypothesis there are no $(e,i)$-hooks in $\lambda$, we must have $(a+k(e-i)-h_2,b-ki+1)\in\lambda$ for all $k\leq{k_C}$ and $(a-k(e-i)-h_2,b+ki+1)\in\lambda$ for all $k<k_A$.\newline
Continuing with this (finite) process we arrive to the following scenario: consider  $(a+k_C(e-i)-h_{i+1},b-k_Ci+i)\in\lambda$. Observe that since $k_C$ was taken to be minimal, we have $(a+(k_C-1)(e-i),b-(k_C-1)i)\in\lambda$ and so $a+k_C(e-i)-h_{i+1}\geq{a+(k_C-1)(e-i)}$. Then by our  hypothesis, $(a-(k_A-1)(e-i)-h_{i+1},b+(k_A-1)i+i)\in\lambda$. Again, $a-(k_A-1)(e-i)-h_{i+1}\geq{a-k_A(e-i)}$ since $0<h_{i+1}\leq{e-i}$. But this would imply $A=(a-k_A(e-i),b+k_Ai)\in\lambda$ and this contradicts our initial hypothesis of $A\not\in\lambda$. Thus, $\lambda$ is $(e,i)$-singular. The following diagram represents the construction described in the proof, with the lines passing through rungs in the same $(e,i)$-ladder.

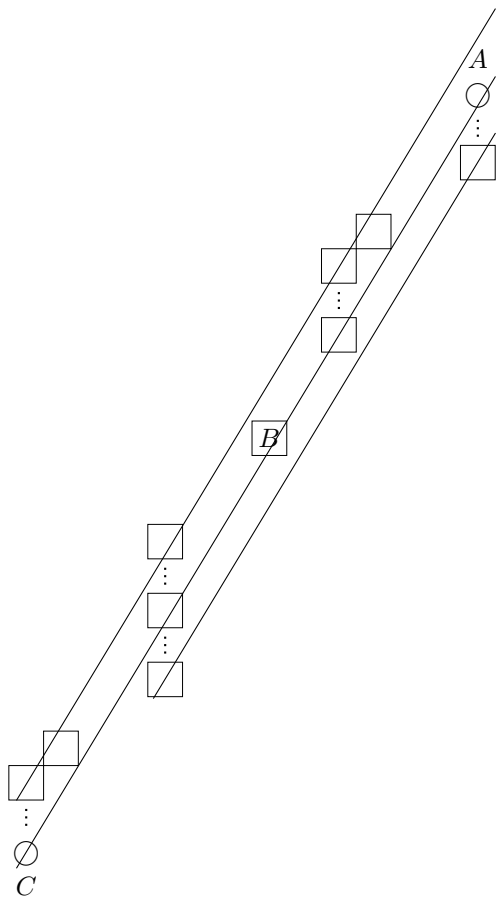
\begin{center}
\begin{tikzpicture}[scale=1]
\tgyoung(0cm,0cm,::::::::::::::<A>,::::::::::::::<\bigcirc>,::::::::::::::\vdts,:::::::::::::;,,::::::::::;,:::::::::;,::::::::::\vdts,:::::::::;,,,:::::::;<B>,,,::::;,:::::\vdts,::::;,:::::\vdts,::::;,,:;,;,:\vdts,:<\bigcirc>,:<C>)
\draw(0.1,-10.5)--(6.4,0);
\draw(0.1,-9.6)--(6.4,0.9);
\draw (1.9,-8.25)--(6.4,-0.75);

\end{tikzpicture}
\end{center}
\end{proof}
\vspace{8mm}

{\Large\textbf{Conclusion and future work}}
\vspace{4mm}

In these notes we have described a family of maps on integer partitions generalising James's $e$-regularisation. For fixed $e$ and $0<i<e$, these maps split the set of partitions for a given integer $n$ in equivalence classes, that we call $(e,i)$-regularisation classes. In Section $3$, we characterised $(e,i)$-regular partitions both via their Young diagram and their abacus display in terms of $(e,i)$-hooks (Lemma~\ref{uniqueness}  and Lemma~\ref{abacus1}, respectively). In Section $4$, we prove an algorithm to construct the $(e,i)$-regularisation of a partition in the abacus, and use it to show in Section $5$ that there is a unique $(e,i)$-regular partition in each class and it dominates every other partition in the class.

The motivation for this work is the study of the extension of James's result on $e$-regularisations to $(e,i)$-regularisations, namely that for a given $\lambda\vdash{n}$, the simple module labelled by $\lambda^{reg_{e,i}}$ when such partition is $e$-regular is a composition factor of $S^{\lambda}$ with multiplicity one in $\He_{\F,q}(\mathfrak{S}_n)$. This will be the matter of [MB], but in light of other uses of these maps already found by Fayers in [F2], we consider the material contained here of interest to be published independently.\newline
In discussions with Alex Fink at Queen Mary University of London, it was suggested to the author that the composition of our $(e,i)$-regularisation with James's one (the $(e,1)$-regularisation), which ensures obtaining a $e$-regular partition, might deliver non-zero composition factors. We have corroborated this conjecture only experimentally with the computational help of GAP, and it remains a question to be explored further.
\vspace{8mm}

{\Large\textbf{Acknowledgements}}
\vspace{4mm}

The author would like to thank Dr Matthew Fayers, his PhD supervisor, for his guidance, patience and support throughout the course of this research, without which this work would not have been possible. 
\vspace{8mm}

{\Large\textbf{References}}
\vspace{4mm}

\begin{tabular}{p{0.75cm}p{13.35cm}}


[DY] & P. Dimakis, G. Yue, \textit{Combinatorial wall-crossing and the Mullineux involution}, J Algebr Comb 50 (2019), 49--72.\\

& \\

[F1] & M. Fayers, \textit{Regularising a partition on the abacus}, 2009, unpublished. Retrieved from \textless http://www.maths.qmul.ac.uk/$\sim$mf/papers/abreg.pdf \textgreater\\

&\\

[F2] & M. Fayers, \textit{Regularisation, crystals and the Mullineux map}, 2021, preprint.\\

&\\

[J1] & G.D. James, \textit{On the decomposition matrices of the symmetric group II}, J. Algebra \text{43}, (1976), 45--54.\\

&\\

[J2] & G.D. James, \textit{The representation theory of the symmetric group}, Lecture Notes in Mathematics, Springer, 1978.\\

&\\

[Ma] & A. Mathas, \textit{Iwahori--Hecke algebras and Schur algebras of the symmetric group}, University lecture series 15, American Mathematical Society, Providence, RI, 1999.\\

&\\

[MB] & D. Millan Berdasco, \textit{A new family of decomposition numbers of Iwahori--Hecke algebras}, in preparation.\\
\end{tabular}
\end{document}